\documentclass[12pt]{amsart}
\usepackage[utf8]{inputenc}
\usepackage[english]{babel}
\usepackage[T1]{fontenc}
\usepackage{amsmath, amssymb}
\newcommand{\Q}{{\QQ}}
\newcommand{\Qtr}{{\mathbb Q}^{\rm tr}}
\newcommand{\Qab}{{\mathbb Q}^{\rm ab}}
\newtheorem{question}{Question}
\newtheorem*{remark*}{Remark}
\newtheorem*{problem*}{Question}
\newtheorem{problem}{Problem}
\newtheorem{algo}{Algorithm}
\def\ds{\displaystyle}
\renewcommand\atop[2]{{\substack{#1\\ #2}}}

\numberwithin{equation}{section}
\usepackage[a4paper,top=3cm,bottom=2cm,left=3cm,right=3cm,marginparwidth=1.75cm]{geometry}

\usepackage{amsthm}
\usepackage{amsfonts}
\usepackage{graphicx}
\usepackage{hyperref}
\usepackage{url}
\newtheorem{theorem}{Theorem} [section]

\newtheorem{definition}[theorem]{Definition}
\newtheorem{lemma}[theorem]{Lemma}
\newtheorem{proposition}[theorem]{Proposition}
\newtheorem{example}[theorem]{Example}
\newtheorem{examples}[theorem]{Examples}

\theoremstyle{remark}
\newtheorem{remark}[theorem]{Remark}

\def\CC{\mathbb{C}}
\def\C{\mathbb{C}}

\def\NN{\mathbb{N}}
\def\QQ{\mathbb{Q}}
\def\RR{\mathbb{R}}
\def\ZZ{\mathbb{Z}}
\def\vv{\mathbf{v}}
\def\ww{\mathbf{w}}

\def\yy{\mathbf{y}}
\def\s{\sigma}
\def\Qalg{\overline{\QQ}}
\def\z{\zeta}

\def\calH{\mathcal{H}}
\def\calL{\mathcal{L}}
\def\calM{\mathcal{M}}
\def\calo{\mathcal{O}}
\def\calR{\mathcal{R}}
\def\calY{\mathcal{Y}}

\def\OK{\calo_K}
\def\UK{\calo_K^\times}

\def\Gal{\mathrm{Gal}}

\def\Z{\mathbb{Z}}
\def\sp{\mathrm{span}}
\def\frA{\mathfrak{A}}

\DeclareMathOperator{\PB}{(B)} 
\DeclareMathOperator{\PS}{(S)}

\title{Large algebraic integers}
\author{Denis Simon}
\address{Normandie Univ, UNICAEN, CNRS, Laboratoire de Math\'ematiques Nicolas Oresme, 14000 Caen, France}
\email{denis.simon@unicaen.fr}
\author{Lea Terracini}
\address{Dipartimento di Informatica, Università di Torino}
\email{lea.terracini@unito.it}
\thanks{Lea Terracini is member of the Italian INdAM group GNSAGA.}

\begin{document}

\maketitle
\begin{abstract}
An algebraic integer is said \emph{large} if all
its real or complex embeddings have absolute value larger than $1$. An integral ideal is said \emph{large} if it admits a large generator. 
We investigate the notion of largeness, relating it  to some arithmetic invariants of the field involved, such as the regulator and the covering radius of the lattice of units. We also study its connection with the Weil height and the Bogomolov property. We provide an algorithm for testing largeness and give some applications to the construction of floor functions arising in the theory of continued fractions.
\end{abstract}
\section{Introduction}
Let $K$ be a number field, $\calo_K$ be its ring of integers and $\frA\subseteq\calo_K$ be a principal ideal.
Our aim is to investigate the following property of $\frA$:

\begin{equation}\label{eq:property}\hbox{ $\frA$ admits a generator $x$ such that $|\sigma(x)|\geq 1$ for every embedding $\s: K\to\CC$}.
\end{equation}

We shall call such a generator a \emph{large} element of $\calo_K$ (\emph{strictly large} when the strict inequality holds), and we shall call \emph{large} every ideal satisfying property \eqref{eq:property}.

We shall see that all but finitely many principal ideals $\frA$ are strictly large; in particular this happens  when the \emph{logarithmic norm} $n(\frA)=\frac {N(\frA)}{[K:\QQ]}$ exceeds the \emph{covering radius} of the lattice of units with respect to the $L_\infty$ norm. Moreover, every non-trivial ideal becomes large in a suitable finite extension of $K$.

It is possible to relate the notion of largeness of a principal ideal to the Weil height of its generators. Therefore, lower bounds of the Weil height on $K$, as given by the Bogomolov property, may help to prove that some ideals are not large. 
To this aim, we shall state some inequalities concerning 
the covering radius, the regulator  and  the Weil height of systems of multiplicatively independent units of $K$. We shall apply the technique described above in some concrete example.

As soon as the group of units of $K$ is known, it is relatively easy to decide if a principal ideal is large, and we give the corresponding algorithm. Then, we shall present
the results of applying it to some particular ideal in cyclotomic fields.

As a last application, we shall define the notion of floor function for $K$ relatively to $\frA$ and show that condition \eqref{eq:property} allows to explicitly construct a bounded floor function.
This turns out to be a good property for the the resulting continued fractions (\cite{CapuanoMurruTerracini2022}).

\section{The largeness property and general results}\label{sect:theproblem}
Let $K$ be a number field of degree $d$.
We denote by $\OK$ the ring of integers of $K$ and $\UK$ the group of units.
The number field $K$ has $r_1$ real embeddings $\s_1,\dots,\s_{r_1}$
and $2r_2$ complex embeddings $\s_{r_1+1},\dots,\s_{r_1+2r_2}$,
where $r_1+2r_2 = d$ and $\s_{r_1+i}$ and $\s_{r_1+r_2+i}$
are conjugates for $1\leq i \leq r_2$. We denote by $\Sigma$ the whole set of Archimedean embeddings. We shall denote by $|\cdot |$ the standard complex absolute value.
An element $x \in K$ therefore has
$r_1+r_2$ Archimedean absolute values, namely $|\s_1(x)|, \dots, |\s_{r_1+r_2}(x)|$,
and we have $|\s_{r_1+i}(x)| = |\s_{r_1+r_2+i}(x)|$
for all $1\leq i \leq r_2$.
We put $s=r_1+r_2$, $r=s-1$.

Let
\begin{align*} \iota: K& \longrightarrow  \RR^{r_1}\times \CC^{r_2}\\
\lambda &\longmapsto (\sigma_1(\lambda),\ldots,\sigma_{r_{1}}(\lambda),\sigma_{r_1+1}(\lambda),\ldots,\sigma_{r_1+r_2}(\lambda))
\end{align*} be the canonical embedding of $K$, and
$$\ell: K^\times \to \RR^{r_1+r_2}$$ be the logarithmic embedding, i.e., the composition $\calL\circ\iota$ where 
\begin{align*} \calL: (\RR^\times)^{r_1}\times (\CC^\times)^{r_2}& \longrightarrow  \RR^{r_1}\times \RR^{r_2}\\
(x_1,\ldots, x_{r_1}, y_1,\ldots, y_{r_2})  &\longmapsto (\log|x_1|,\ldots, \log|x_{r_1}|, 2\log|y_1|,\ldots, 2\log|y_{r_2}|).
\end{align*} \\
For $\mathbf{x}=(x_1,\ldots, x_{r_1},y_1,\ldots, y_{r_2})\in \RR^{r_1}\times \CC^{r_2}$, let us define
$$N(\mathbf{x})=\prod_{i=1}^{r_1}|x_i 
|\cdot \prod_{j=1}^{r_2}|y_j|^2;$$ 
then, $N(\iota(a))=|N_{K/\QQ}(a)|$ for every $a\in K.$
The {\em absolute norm} of $x\in K$ is equal to the absolute value of the norm of $x$.\\
We shall denote $\Lambda_K=\ell(\calo_K^\times)$; it is a lattice of rank $r$ in $\RR^s$ by Dirichlet's Unit Theorem. We recall also that  the \emph{regulator} $R_K$ of $K$ is the determinant of any submatrix of order $r$ of the $r\times (r+1)$ matrix whose rows are $\ell(u_1),\ldots, \ell(u_r)$ 
for a system $u_1,\ldots, u_r$ of fundamental units for $K$. If $V_K$ is  the volume  of a fundamental domain for $\Lambda_K$ then the relation $V_K=\sqrt{s}R_K$ holds.

\begin{definition}\ \begin{itemize}
    \item[a)] We say that $x\in \calo_K$ is \emph{large} (resp. \emph{strictly large}) if
$|\sigma(x)|\geq  1$ (resp. $|\sigma(x)|> 1$) for every $\sigma\in\Sigma$ . 
\item[b)] An ideal $\frA\subseteq\calo_K$ is \emph{large} (resp. \emph{strictly large}) if it principal and has a large (resp. strictly large) generator.
\end{itemize}  
\end{definition}

For $x \in \calo_K$
as in the above definition,
we observe that
$x$ is large in $\calo_K$ if and only if $x$ is large in $\calo_L$ for any finite extension $L$ of $K$.
For the ideal $\frA$
it is true that if it is large in $K$, then
the ideal $\frA \calo_L$ is large in $\calo_L$, but the converse is not true.

Then the following definition makes sense and extends the notion of largeness to possibly infinite extensions:
\begin{definition}
Let $L$ be an infinite extension of $K$.
For an element $x \in \calo_K$
and an ideal $\frA \subseteq \calo_K$, we say that:
\begin{itemize}
\item[a)]
$x$ is \emph{large} (resp. \emph{strictly large}) in $L$ if there is a number field $K'$ with $K \subseteq K'\subseteq L$ such that $x$ is large (resp. strictly large) in $K'$.
\item[b)]
$\frA$ is \emph{large} (resp. \emph{strictly large}) in $L$ if there is a number field $K'$ with $K \subseteq K'\subseteq L$ such that $\frA \calo_{K'}$ is large (resp. strictly large) in $K'$.
\end{itemize}
\end{definition}

Since every unit in $\calo_K$ has norm $\pm 1$, a unit $x$ is large if and only if $|\sigma(x)|=1$ for every $\sigma\in \Sigma$; by a theorem of Kronecker \cite{Kronecker1857}, this happens if and only if $x$ is a root of unity. Therefore no unit can be strictly large. So every strictly large element $x$ in $\calo_K$ must satisfy $|N_{K/\QQ}(x)|\geq 2$.\\
We shall see in Proposition \ref{prop:normalarge} that almost all principal ideals in $\calo_K$ are (strictly) large. In order to prove this fact we need to recall some terminology from lattice theory. 

Let $\Lambda$ be a lattice in $\RR^n$ of rank $r$ and for a real number $p\in [1,\infty)\cup\{\infty\}$ let $|| \cdot||_p$ be the norm $L_p$ in $\RR^n$. The \emph{distance function} relatively to $p$ is by definition
    $$\rho_p(\vv,\Lambda)=\min_{\ww\in\Lambda} ||\vv-\ww||_p.$$
    The \emph{covering radius} of $\Lambda$ with respect to $|| \cdot||_p$ is
    $$\rho_p(\Lambda)=\sup_{\vv\in \sp(\Lambda)} \rho_p(\vv,\Lambda).$$
     Balls of radius $\rho_p(\Lambda)$ centered around all lattice points cover the whole space $\sp(\Lambda)$.\\
    By the well known inequality
     \begin{equation}\label{eq:disuglp} ||\vv||_p\leq ||\vv||_r\leq n^{\frac 1 r-\frac 1 p}||\vv||_p\quad \hbox{for } \infty\geq p\geq r,\end{equation}
     we get 
    \begin{equation}\label{eq:disugcovrad} \rho_p(\Lambda)\leq \rho_r(\Lambda) \leq n^{\frac 1 r-\frac 1 p}\rho_p(\Lambda)\quad \hbox{ for } \infty\geq p\geq r.\end{equation} 
    If $K$ is a number field we shall write $\rho_p(K)$ instead of $\rho_p(\Lambda_K)$.
    
    For every algebraic number $x\in\overline{\QQ}^\times$ we define the \emph{logarithmic norm}
    $$n(x)=\frac{\log|N_{\QQ(x)/\QQ}(x)|}{[\QQ(x):\QQ]}.$$
    Analogously, if $\frA\subseteq \calo_K$ is any non-zero ideal, we write
    $$n(\frA)=\frac{\log|N_{K/\QQ}(\frA)|}{[K:\QQ]}.$$
    Notice that $$n(x)=\frac{\log|N_{K/\QQ}(x)|}{[K:\QQ]},$$
    for every finite extension $K$ of $\QQ(x)$; moreover
    $$n(x)=n(ux),$$
    for every  algebraic unit $u\in\overline{\QQ}$. 
    Then $n(a)=n(a\calo_K)$
     depends only on the principal ideal generated by $a$ in the ring of integers of every number field containing $a$.\\
     We also observe that $n:\Qalg^\times\to \RR$ is a morphism; in particular $n(x^k)=kn(x)$  for every $k\in\NN$.
     We also have $n(x) \ge 0$ when $x$ is an algebraic integer.
\begin{proposition}\label{prop:normalarge}  
 Every principal ideal $\frA$ of $\calo_K$ such that $n(\frA)>\rho_\infty(K)$ is strictly large.\\
 Therefore all but finitely many integral principal ideals of $\calo_K$ are strictly large.
\end{proposition}
\begin{proof}
Let $x\in \calo_K$ be a generator of $\frA$ and put $N=|N_{K/\QQ}(x)|$. The image of units $\ell(\calo_K^\times)$ is a lattice in  $\RR^s$ of rank $r=s-1$; it spans the hyperplane $\calH$ of $\RR^s$ with equation $\sum_{i=1}^{r_1} x_i+\sum_{i=1}^{r_2}y_i=0$. The vector $$\yy=\ell(x)-\frac 1 d \log(N)(1,\ldots,1,2,\ldots,2)$$ lies on $\calH$ . Let $\rho=\rho_\infty(K)$ and assume $n(x)>\rho$; by definition of covering radius, there exists $u\in\calo_K^\times$ such that $||\yy+\ell(u)||_\infty \leq \rho$. This means that $|\log|\sigma(ux)|-n(x)|\leq \rho$ for every Archimedean embedding $\sigma$ of $K$, so that
$$|\log|\sigma(ux)||\geq n(x)-\rho> 0.$$
The second assertion follows from the fact that the ideals of norm $\le\rho$ are finitely many.
\end{proof}

\begin{proposition}\label{prop:esisteL} Every non-trivial integral ideal $\frA \subsetneq \calo_K$ is strictly large in $\overline{\QQ}$.
\end{proposition}
\begin{proof} The statement is a consequence of \cite[Théorème 5.1]{BergeMartinet1989}, here we present a more direct proof.
First of all, by class field theory, it is well known that $\frA$ becomes principal in a suitable finite extension $K'$ of $K$.
We have $\frA \calo_{K'} = x \calo_{K'}$ for some $x \in \calo_{K'}$.
By Proposition \ref{prop:normalarge} there exists a positive integer $j$ such that $x^j$ is strictly large in $K'$. Let $u\in \calo_{K'}^\times$ be such that $|\sigma(u x^j)|>1$ for every embedding $\sigma:K'\to\CC$. Let $L=K'(\omega)$ where $\omega^j=u$. Let $\tau:L\to\CC$ be any embedding and $\sigma$ be the restriction of $\tau$ to $K'$. Then
$$|\tau(\omega x)|^j=|\sigma(ux^j)|>1$$
so that $|\tau(\omega x)|>1$.
\end{proof}

By looking at the proof of Proposition \ref{prop:esisteL}, we see that a uniform and stronger version holds true.
For every number field $K$ and every positive $j\in\NN$, we denote by $K_j$ the field obtained from $K$ by adding the $j$-th roots of every unit of $K$; it is a finite extension of $K$ by Dirichlet's Unit Theorem.
\begin{proposition}\label{prop:esisteLuniforme}
Let $K$ be a number field, and let $j> \frac{\rho_\infty(K)[K:\QQ]}{\log2}$. Every non-trivial principal ideal $\frA \subsetneq \calo_K$ is strictly large in $K_j$.
  \end{proposition}
  \begin{proof}
  Let $x \in \calo_K$ be a generator of $\frA$.
  We have $n(x)\geq \frac{\log 2}{[K:\QQ]}$, so that $n(x^j)=jn(x) >\rho_\infty(K)$. Then one can choose $L=K_j$ in the proof of Proposition \ref{prop:esisteL}.
  \end{proof}
\section{Largeness and Weil height}\label{sect:height}
Let $h$ denote the logarithmic Weil height of an algebraic number (see for example \cite[\S 1.5.7]{BombieriGubler2006}). For $x\in K$
$$h(x)=\frac 1 d\sum_{\sigma\in \Sigma} \max\{0,\log|\sigma(x)|\}+\log|a|$$
where $a$ is the leading coefficient of a primitive equation for $x$ over $\ZZ$; in particular
for an algebraic integer $x$ in $\calo_K$
$$h(x)=\frac 1 d\sum_{\sigma\in \Sigma} \max\{0,\log|\sigma(x)|\}.$$
It follows that 
\begin{equation}\label{eq:h>n} h(x)\geq \frac 1 d\log|N_{K/\QQ}(x)|=n(x) \quad \hbox{ for every non-zero algebraic integer $x$},\end{equation}
and equality holds exactly when $x\calo_K$ is large.\\
Then we can draw necessary conditions for largeness of ideals when some explicit minoration for the height of elements in $\calo_K$ is known. Namely, if there is a constant $c>0$ such that 
\begin{equation}\label{eq:maggiorazioneperinteri} h(x)>c\hbox{ for every } x\in\calo_K\setminus\calo_K^\times\end{equation}
and  $\frA$ is  a principal ideal such that $n(\frA)\leq c$, then $\frA$ cannot be large.\\
We are thus lead to make use of the well known \emph{Bogomolov property} $\PB$  and an additional property $\PS$ defined below.\\
Let $\mathcal{A}$ be a set of algebraic numbers.
  We put
    \begin{align*}
    b(\mathcal{A})&=\inf\{h(x)\ |\ x\in \mathcal{A}, x\not=0, x\hbox{ not a root of unity }\};\\
    s(\mathcal{A})&=\inf\{n(x)\ |\ x\in\mathcal{A}, x\not=0, N_{\QQ(x)/\QQ}(x)\not=\pm 1\}.
    \end{align*}
\begin{definition}\label{def:propBS} We say that a set $\mathcal{A}$ of algebraic numbers satisfies
 \begin{itemize} 
    \item[a)]   \emph{property $\PB$} if $b(\mathcal{A})>0$;
\item[b)]  \emph{property $\PS$}  if $s(\mathcal{A})>0$.
\end{itemize}
\end{definition}

In particular, if $x\in\calo_L\setminus \calo_L^\times $ for  an (infinite) extension $L$  with the property $\PB$, then $h(x)\geq c_L$ for some $c_L>0$ depending only on $L$ and thus, by \eqref{eq:h>n}, 
$$x\calo_{\QQ(x)} \hbox{ large} \Longrightarrow 
n(x)\geq c_L.
$$
Property $\PB$ is known for some special algebraic extensions, as the compositum $\Qtr$ of all totally real fields; it is also known for extensions having bounded local degrees at some finite place, and for Abelian extensions of number fields (see~\cite[Remark 5.2, p.1902]{AmorosoDavidZannier2014}). \\
Note however that property $\PB$ for the whole field  $L$, and even for the ring of integers of $L$, is much stronger that condition \eqref{eq:maggiorazioneperinteri}, which assumes a lower bound only for the height of algebraic \emph{integers} which are not units. 

\begin{examples}\label{exs:esempi}~

\begin{itemize}
\item[a)] Of course $\PS\Rightarrow \PB$ if $\mathcal{A}$ is a set of algebraic integers containing only a finite number of units. 
 
\item[b)] On the other hand, there exist sets of algebraic integers satisfying 
$\PB$ but not $\PS$: for example   the ring $\calo^{\rm ab}$ of integers of $\QQ^{\rm ab}$ satisfies property $\PB$ with $b(\calo^{\rm ab})\geq \frac{\log 5}{12}$, (see the main theorem in \cite{AmorosoDvornicich2000}), but  
$$n(1-\zeta_p)=\frac{\log(p)}{p-1}$$ 
for a prime $p$ and $\zeta_p$ a primitive $p$-th root of unity; therefore $s(\calo^{\rm ab})=0$.
\item[c)] It is  proven in \cite[Corollary 1]{AmorosoDvornicich2000} that  property $\PS$ holds for the set $\mathcal{A}$ of algebraic integers $x$ lying in an Abelian extension of $\QQ$ and such that $  x/{\overline{x}}$ is not a root of unity. More precisely 
$$n(x)\geq \frac{\log 5 } {12},\hbox{ for every $x\in\mathcal{A}$ }.$$ 
\item[d)] Recall that $\Qtr(i)$ is the compositum of all CM fields, see \cite[page 1902]{AmorosoDavidZannier2014}; therefore $x\in \Qtr(i)$ if and only if  $\QQ(x)$ is either a totally real or a CM field. Since the complex conjugation commutes with all the embeddings of $\Qtr(i)$ in $\CC$, 
we have $|\sigma(x)| = 1$ for some $\sigma \in \Sigma$ if and only if 
$|\sigma(x)| = 1$ for all $\sigma \in \Sigma$.
In this case, we just write $|x| = 1$.\\
By a result of Schinzel (apply~\cite[Corollary 1', p. 386]{Schinzel1973}, to the linear polynomial $P(z)=z-x$), if
$\mathcal{A}=\{x\in \Qtr(i)\ |\ |x|\not=1\}$ then 
\begin{equation}\label{eq:Schinzel} b(\mathcal{A})\geq \frac 1 2\log \frac{1+\sqrt{5}} 2.\end{equation}
\end{itemize}
\end{examples}

By Example \ref{exs:esempi}.d) we obtain the following
\begin{proposition}\label{prop:notlargeperSchinzel}
Let  $L=\Qtr(i)$, and let $x\in\calo_L $ be a non-zero element.
If $n(x)< \frac 1 2 \log\frac{1+\sqrt{5}}2$ then $x \calo_{\QQ(x)}$ is not large in $L$ except if $x$ is a unit.
\end{proposition}
\begin{proof}
If $x$ is a unit, then $x\calo_{\QQ(x)} = \calo_{\QQ(x)}$ is trivially large. If not, we have $|x|\not=1$, so that Schinzel result \eqref{eq:Schinzel} implies that $ h(x)\geq \frac 1 2\log \frac{1+\sqrt{5}} 2>n(x)$. Then the result follows from \eqref{eq:h>n}.
\end{proof}

\subsection{Example}  Let $p$ be one of the primes for which $\Q(\zeta_{p-1})$ has class number one. Note that $p$ splits completely in $\Q(\zeta_{p-1})$. Recall (\cite{MasleyMontgomery1976}) that the cyclotomic field $\Q(\zeta_m)$ has class number one if and only if $m$ is one of the following forty-four numbers:
\begin{multline*}
3, 4, 5, 6, 7, 8, 9, 10, 11, 12, 13, 14, 15, 16, 17, 18, 19, 20, 21, 22, 24, 25, 26, 27, 28,\\ 
30, 32, 33, 34, 35, 36, 38, 40, 42, 44, 45, 48, 50, 54, 60, 66, 70, 84, 90.
\end{multline*}
(which corresponds to twenty-nine distinct cyclotomic fields). Thus the relevant primes are
\begin{equation}
\label{list}
5, 7, 11, 13, 17, 19, 23, 29, 31, 37, 41, 43, 61, 67, 71.
\end{equation} 
\begin{question}
\label{qu:continued}
Let $p$ be one of the fifteen primes~\eqref{list} and let $\mathfrak{P}$ be a prime ideal over $p$ in the ring of integers of $\QQ(\zeta_{p-1})$. Is $\mathfrak{P}$ large?
\end{question}
Since $\Qab\subseteq\Qtr(i)$ and
$$
p\leq \Big(\frac{1+\sqrt{5}}2\Big)^{\varphi(p-1)/2}\hbox{ for }p=41,67\hbox{ and }71
$$
the answer to Question~\ref{qu:continued} is negative for these primes, by Proposition \ref{prop:notlargeperSchinzel}. 

Note that we could have tried the same strategy as in Proposition \ref{prop:notlargeperSchinzel}
but using the inequality of Example \ref{exs:esempi}b. This would work only for the primes $p$ satisfying
\begin{equation}\label{eq:valetutti}
p < 5^{\varphi(p-1)/12}.
\end{equation}
However, this inequality is satisfied by none of the primes in the list~\eqref{list}.

In the subsequent Theorem \ref{teo:risposta},  Question \eqref{qu:continued} will receive a complete answer.
 \subsection{Number theoretic minorations of the covering radius}
 In the light of the propositions \ref{prop:normalarge} and \ref{prop:esisteLuniforme}, it is useful to have some quantitative information on the covering radius $\rho_\infty(K)$ of a number field $K$ of degree $d$. \\
 Let $\lambda_1,\ldots,\lambda_r$ be the successive minima (w.r.t. the Euclidean norm $||\cdot ||_2$) of the lattice $\Lambda_K$.  
 It is well known (see for example \cite[Theorem 7.9]{MicciancioGoldwasser2002} that 
\begin{equation}\label{eq:disugminsucccovrad} \lambda_1\leq\ldots\leq \lambda_r \leq 2\rho_2\leq \sqrt{s}\lambda_r.\end{equation}
Moreover by \eqref{eq:disugcovrad} we have
 \begin{equation}\label{eq:disugcovrad2}
 \rho_\infty(K)\leq \rho_2(K)\leq  \sqrt{s}\rho_\infty(K).\end{equation}
 Therefore 
 \begin{align} \rho_\infty(K)&\geq \frac 1 {\sqrt{s}}\rho_2(K)\quad\hbox{from \eqref{eq:disugcovrad2}}\nonumber \\\
& \geq \frac 1 {2\sqrt{s}}\lambda_r\quad\hbox{from \eqref{eq:disugminsucccovrad}}\label{eq:keyminoration}
 \end{align} 
 
 \begin{theorem}\label{teo:regolatore}~
  \begin{itemize}
     \item [a)] Let $K$ be a number field such that $r\geq 1$. Then  $\rho_\infty(K)\geq \frac 1 2 R_K^{\frac 1r}$. 
     \item[b)] 
  There exists a constant $c>0$ such that $\rho_\infty(K)\geq c$ for every number field $K$ such that $r\geq1$.
  \end{itemize}
 \end{theorem}
 
 \begin{proof}
 Recall that $V_K$ is the volume of a fundamental domain for $\Lambda_K$.
 By Minkowski's Second Theorem \cite[Theorem 1.5]{MicciancioGoldwasser2002}
 $$(\lambda_1\cdot\ldots\cdot \lambda_r)^{\frac 1 r}\geq \sqrt{r}\cdot  V_K^{\frac 1r}= \sqrt{r}\cdot (\sqrt{s}R_K)^{\frac 1r}.$$
 Then from \eqref{eq:keyminoration}
 \begin{align}\label{eq:covradreg} \rho_\infty(K)\geq \frac {\sqrt{r}}{2\sqrt{s}}(\sqrt{s}R_K)^{\frac 1r}\geq c' R_K^{\frac 1r},\end{align}
 for a suitable constant $c'$.  Studying the function $\frac {\sqrt{r}}{2\sqrt{s}}(\sqrt{s})^{\frac 1 r}$ with $s=r+1$ we see that $c'=\frac 1 2$. This proves a). Then b)  follows from the well known fact that there exist constants $c_0>0$ and $c_1>1$ such that $R_K>c_0\cdot c_1^d$ (\cite[\S 3]{Zimmert1981}, see also \cite{FriedmanSkoruppa1999}. \end{proof}
 
 The next results provide some lower bounds for the covering radius involving the Weil height on $K$.\\
For $n=1,...,r$ we put, as in \cite[Page 9]{AmorosoDavid2021}
$$\mu_K(n)=\inf_{\atop{v_1,\ldots, v_n\in\calo_K^\times}{{\rm multipl. indep.}}}(h(v_1)\cdot\ldots\cdot h(v_n)),$$
 
 \begin{theorem}
 \label{teo:eccolo}
 For $n=1,\ldots,r$, we have
 $$\rho_\infty(K)\geq \frac d s\mu_K(n)^{\frac 1 n}.$$
 In particular $$\rho_\infty(K)\geq \frac d s b(\calo_K^\times).$$
 \end{theorem}
 \begin{proof}
For every $u\in\calo_K^\times$, by \eqref{eq:disuglp},
\begin{equation}\label{eq:due} 2dh(u)= ||\ell(u)||_1 \leq  \sqrt{s}||\ell(u)||_2.\end{equation} 
 
There exist multiplicatively independent $u_1,\ldots, u_r \in\calo_K^\times$ such that  $\lambda_i=||\ell(u_i)||_2$ (\cite[Theorem 1.2]{MicciancioGoldwasser2002}). 
We notice that for $n=1,\ldots, r$
\begin{align*} \lambda_n & =\inf_{\atop{v_1,\ldots, v_n\in\calo_K^\times} {\rm multipl. indep.}}  \max\{ ||\ell(v_1)||_2,\ldots,  ||\ell(v_n)||_2\}\\
&\geq \frac {2d} {\sqrt{s}} \inf_{\atop{v_1,\ldots, v_n\in\calo_K^\times} {\rm multipl. indep.}} \max \{ h(v_1),\ldots,  h(v_n)\}\quad\hbox{ by \eqref{eq:due}.}
\end{align*}
It follows from \eqref{eq:keyminoration} that
\begin{align*} \rho_\infty(K) &\geq \frac 1 {2\sqrt{s}} \lambda_n\geq \frac d s\inf_{\atop{v_1,\ldots, v_n\in\calo_K^\times} {\rm multipl. indep.}}  \max \{ h(v_1),\ldots,  h(v_n)\}\geq \frac d s\mu_K(n)^{\frac 1 n}.\end{align*}
\end{proof}

By inequality \eqref{eq:covradreg}, it is possible to use known lower bounds for the regulator $R_K$ in order to deduce lower bounds for the covering radius $\rho_\infty(K)$. See \cite[\S 3.5, 15]{Narkiewicz2004} for an overview on evaluations of the regulator and \cite[\S 8]{Narkiewicz2004} for the special case of Abelian extensions. Moreover \cite[Proposition 3.3]{AmorosoDavid2021} provides a tool allowing to improve the bound from below of extensions $K$ for which $\calo_K^\times$ has property $\PB$. In particular \cite[Corollaire 3.5]{AmorosoDavid2021} deals with the case of totally real and CM field.

Some remarkable results are collected below:
 
\begin{itemize}
\item[a)] Let $L$ be an infinite extension. Assume that $\calo_L^\times$ satisfies property $\PB$ and let $c_L= b(\calo_L^\times)=\inf_{u\in\calo_L^\times\setminus\calo_L^{\times,\rm{tors}}} h(u) >0$. Then by Theorem \ref{teo:eccolo}$$ \rho_\infty(K)\geq c_L,$$ for every number field $K \subseteq L$ such that $r(K) \geq 1$.
\item[b)] In particular, if a number field $K$ is contained in $\QQ^{\rm{tr}}(i)$, then by Example \ref{exs:esempi} d),
$$\rho_\infty(K)\geq \frac 1 2 \log \frac{1+\sqrt{5}} 2.$$
\item[c)] Silverman's theorem \cite{Silverman1984} allows to construct fields with  fixed degree and covering radius arbitrarily large. It suffices to choose a non-CM field of discriminant large enough.  
\end{itemize}

\section{An algorithm for largeness}\label{sect:algorithm}

  We describe an algorithm that solves the following problem:
  \begin{problem}\label{problem1}
  Given a non-zero algebraic number $x\in K$ and a bound $B>0$,
  find all units $u\in \UK$ such that
  $|\s(xu)| \geq  B$ for all $\s \in \Sigma$.
\end{problem}
When $B=1$ this algorithm detects when a principal ideal is large.
\subsection{Basic results}

\begin{lemma}\label{upperbound}
  If $u \in \UK$ is a solution of Problem \ref{problem1},
  then for all $S \subset \Sigma$, we have
  $$B^{\# S} \leq \prod_{\s \in S} |\s(xu)| \leq B^{\# S - d} N(x)$$
  where $\# S$ denotes the cardinality of $S$.
\end{lemma}
\begin{proof}
Let $u \in \calo_K^\times$ be a solution of Problem \ref{problem1}.
For $\s \in S$, we have the trivial inequality $B \leq |\s(xu)|$.
Multiplying these inequalities gives the announced left inequality.

For $\s \in S$, we have the inequality $|\s(xu)| \leq |\s(xu)|$, and
for $\s \not\in S$, we have $B \leq |\s(xu)|$.
Multiplying these inequalities gives $B^{d-\# S} \prod_{\s \in S} |\s(xu)| \leq N(xu)$.
But $u$ is a unit, hence $N(xu) = N(x)$, whence the result.
\end{proof}

\begin{proposition}\label{smallnorm}
  If $N(x) < B^d$, then Problem \ref{problem1} has no solution.
\end{proposition}
\begin{proof}
Apply Lemma \ref{upperbound} with $S = \Sigma$.
\end{proof}

\begin{proposition}\label{finitesol}
  Given $x\in K$ and $B>0$, Problem \ref{problem1} has only finitely many solutions.
\end{proposition}
\begin{proof}
By Proposition \ref{smallnorm}, Problem \ref{problem1} has no solution for $x=0$.
We assume now that $x\neq 0$.

Dividing the right inequality of Lemma \ref{upperbound} by $\prod_{\s \in S}|\s(x)| \neq 0$ gives
$$ \prod_{\s\in S} |\s(u)| \leqslant  B^{\# S - d} \prod_{\s \not\in S}|\s(x)|
\leqslant B^{\#S-d} \left( \frac {||x||_1}{d-\# S} \right)^{d-\# S} \; .$$
Let us consider the characteristic polynomial of $u$ for the extension $K/\Q$ and
denote it by $P_u$.
Since $u$ is a unit, $P_u \in \ZZ[X]$ and $P_u$ is monic.
The roots of $P_u$ in $\C$ are the real or complex numbers $\s(u)$.
Using the above inequality and expressing the coefficient $a_k$ of $X^k$ in $P_u$ in terms of the roots of $P_u$,
we deduce that $|a_k|$ is bounded independently of $u$.
For example we have
$|a_d| = |a_0|=1$ since $u$ is a unit
and 
$$|a_k| \leq \binom{d}{k} \left( \frac {||x||_1}{kB} \right)^k$$
for the other values of $k$.
Since this bound does not depend on $u$, there are only finitely many
possibilities for $P_u$, hence for $u$.
\end{proof}

\noindent{\bf Remark.}
We could turn the proof of Proposition \ref{finitesol} into an algorithm
that tests all polynomials with coefficients within some bounds depending on $B$ and $x$.
Explicitly, using the bounds given during the proof,
we see that, for a given number field of fixed degree $d$,
the number of polynomials that need to be tested, is proportional to
$\ds \left( \frac {||x||_1}{B} \right)^\alpha$
with $\ds \alpha = \sum_{k=1}^{d-1} k = \frac {d(d-1)}2$.
The number of polynomials that need to be tested in the algorithm
is therefore exponential in the input $x$, hence very large,
and the resulting algorithm is very slow.

We will give another algorithm in the next section.

\subsection{An algorithm to solve Problem \ref{problem1}}
If $A=(a_{i,j})$ is a matrix (or a vector) with real entries,
we write $A \geq 0$ to indicate that $a_{i,j} \geq 0$ for all $i$ and $j$.
We also write $A\geq B$ if $A-B \geq 0$.
We will use the fact that, if  $A\geq 0$ and $B\geq 0$, then $AB \geq 0$.

We recall that, for a number field $K$ of degree $d$, with $r_1$ real embeddings and $r_2$ complex embeddings, we have set $s = r_1+r_2$ and $r = s-1$.
If $u_1,\dots,u_r$ are generators of $\UK$ modulo torsion,
we define the matrix $L$ of size $r\times s$ by
$$ L_{i,j} = \log | \s_j(u_i)|$$
if $\s_j$ is real and 
$$ L_{i,j} = 2 \log | \s_j(u_i)|$$
if $\s_j$ is complex.
The $i$-th row of $L$ is equal to $\ell(u_i)$.
At last, we define the column vector $V = (1,\dots,1)^t \in \Z^s$.

\medskip
Using logarithms, we can reformulate our Problem \ref{problem1} as:

\begin{problem}\label{problem2}
  Given a non-zero algebraic number $x\in K^\times$ and a bound $B>0$,
  find all rows $U \in \Z^r$ such that
  $$ UL + X \geq 0 $$
  where $X = \ell(x)-\calL(B)$.
\end{problem}

Formulated in this way, we see that Problem \ref{problem2} can be solved by integer linear programming. However, the situation is not generic here, and a simpler algorithm is given below.

\begin{algo}\label{algo}~
  
  Input : $x\in K$, $x\neq 0$, and $B>0$.

  Output : all solutions $U \in \Z^r$ of Problem \ref{problem2}.
  
  {\tt
    \begin{enumerate}
    \item Compute the matrix $L$ of size $r\times s$
      and the column vector $V$ of size $s$ as in the above definition.
    \item Remove from $L$ its last column and call $M$ the inverse of this matrix.
      Concatenate $M$ with the row vector of size $r$ whose all entries are $0$
      and obtain a new matrix $M$ of size $s\times r$.
      
      For the next steps, we use the notation $N_{,j}$ for the $j$-th column of a matrix $N$.
    \item Define the matrix $N^+$ of size $s\times r$, such that, for $1\leq j \leq r$, 
      $N^+_{,j} = M_{,j} - \min_i\{M_{i,j}\} V$.
    \item Define the matrix $N^-$ of size $s\times r$, such that, for $1\leq j \leq r$, 
      $N^-_{,j} = M_{,j} - \max_i\{M_{i,j}\} V$.

    \item Compute the row vector $X = \ell(x) - \calL(B)$.
    \item For all row vector $U \in \Z^r$ in the range
      $ -XN^+ \leq U \leq -XN^-$, test if $UL+X \geq 0$.
      If this is the case, output $U$.
    \end{enumerate}
    }
\end{algo}

\begin{proposition}\label{proofalgo}
  Algorithm \ref{algo} is correct.

  Furthermore, when the number field $K$ is fixed,
  the number of $U$ that need to be tested during step 6
  is at most proportional to $( \log N(x)-  d\log B +1)^r$.

\end{proposition}
\begin{proof}
We follow the algorithm step by step.
\begin{enumerate}
\item By Dirichlet's Unit Theorem, the matrix $L$ constructed in step 1 has rank $r$.
  Since the absolute norm of a unit is equal to $1$, we have $LV = 0$,
  hence $V$ is in the right kernel of $L$.
\item By Dirichlet's Unit Theorem,
  when we remove any column of $L$, the determinant of the remaining square matrix
  is always the same and equals the regulator of $K$, which is not $0$.
  This matrix of size $r \times r$ is invertible.
  By construction, we have $LM = I_r$.
\item For all columns of $N^+$, we have $N^+_{,j} = M_{,j} - \min_i\{M_{i,j}\} V$.
  Let $i(j)$ be the index such that $\min_i\{M_{i,j}\} = M_{i(j),j}$.
  We have $N^+_{i,j} = M_{i,j} - M_{i(j),j} \geq 0$ by minimality.
  Hence $N^+ \geq 0$.
  Because $V$ is in the right kernel of $L$, we deduce that
  $LN^+ = LM = I_r$.
\item Using a similar argument, we can prove that
  $N^- \leq 0$ and $LN^- = I_r$.
\item There is nothing to say here.
\item If $U$ is a solution of Problem \ref{problem2},
  then $UL+X \geq 0$.
  But $N^+ \geq 0$, hence $ULN^+ + XN^+ \geq 0$.
  By the relation $LN^+ = I_r$, we deduce $U \geq -XN^+$.
  By $N^- \leq 0$, we deduce $ULN^- + XN^- \leq 0$, and $U \leq -XN^-$.
\end{enumerate}

In order to bound the number of $U$ tested in step $6$,
we observe that $-XN^+ \leq U \leq -XN^-$.
For the $j$-th entry, this is explicitly
$-XN^+_{,j} \leq U_j \leq -XN^-_{,j}$
hence the number of $U_j$ that need to be tested is at most equal to
$-XN^-_{,j} + XN^+_{,j}+1$. But we have
\begin{align*} -XN^-_{,j} + XN^+_{,j} & = X(M_{,j}-\min_i\{M_{i,j}\}V - M_{,j}+\max_i\{M_{i,j}\}V)\\
& = (\max\{M_{,j}\} - \min\{M_{,j}\}) XV
\end{align*}
We also have $XV = \log N(x)- d\log B$.
If $\log N(x) -d\log B < 0$, we have seen in Proposition \ref{smallnorm} that the problem has no solution. When
$\log N(x) -d\log B \geq 0$, we have
$$-XN^-_{,j} + XN^+_{,j}+1
\leq 
(\max\{M_{,j}\} - \min\{M_{,j}\}+1) (\log N(x) - d \log B + 1)$$
whence a bound for the number of $U$ by
$$ (\log N(x) - d\log B+1)^r \times \prod_j (\max\{M_{,j}\} - \min\{M_{,j}\}+1)$$

\end{proof}

\section{A complete example}\label{sect:cyclotomic}

In this section, we shall give a detailed execution of Algorithm \ref{algo}, which answers Question \ref{qu:continued} for $p=17$.
All computations were done using PARI/gp \cite{PARI2}.

\medskip
Let us consider the $16$-th cyclotomic field $K$ equal to $\Q(\z) = \Q[X]/\Phi_{16}(X)$,
where $\Phi_{16}(X) = X^8+1$.

In this field, we consider $x = -\z^7 - \z^3 + \z^2$.
We have $N(x) = 17$, hence $x$ is a generator of a principal prime ideal above $17$.
We are looking for another generator $x'$ of this principal ideal
such that $|\s(x')| \geq 1$ for all $\s\in\Sigma$. We need to solve Problem \ref{problem2} with $B = 1$.

We follow here the steps of Algorithm \ref{algo}.

\begin{enumerate}
\item
  For this field, we have $d=8$, $r_1 = 0$ and $r_2 = 4$.
  In this case, we have $s = r_1+r_2 = 4$ and $r = 3$.
  
  The units of $K$ are generated by
  $u_0 = \z$, $u_1 = -\z^6 + \z^2 - 1$, $u_2 = \z^2+\z+1$ and $u_3=-\z^6+\z^3-\z$,
  where $u_0$ generates the torsion part and $u_1,u_2,u_3$ generate the free part.
  The matrix $L$ is equal to
  $$L =
  \left(
  \begin{array}{rrrr}
  -1.76274&-1.76274& 1.76274& 1.76274\\
  -0.33031& 2.09306&-2.89946& 1.13671\\
   1.13671&-2.89946&-0.33031& 2.09306
  \end{array} \right)
  $$
  We easily check that $L\begin{pmatrix}1\\1\\1\\1\end{pmatrix} = 0$.
\item
  We have
  $$M = 
  \left( \begin{array}{rrrr}
  -0.46575&-0.29144& 0.07276\\
  -0.17430&-0.07276&-0.29144\\
  -0.07276&-0.36421&-0.21868\\
  0 & 0 & 0
  \end{array} \right)
  $$
  We can check that $LM = I_3$, the identity matrix of order $3$.
\item We have
  $\min(M_{,1}) = -0.46575$,
  $\min(M_{,2}) = -0.36421$,
  $\min(M_{,3}) = -0.29144$.
  This gives
 $$N^+ = \left( \begin{array}{rrr}
 0       & 0.07276 & 0.36421 \\
 0.29144 & 0.29144 & 0       \\
 0.39298 & 0       & 0.07276 \\
 0.46575 & 0.36421 & 0.29144 \\
 \end{array}\right)$$
  We can check that $N^+ \geq 0$ and $LN^+ = I_3$.
\item We have
  $\max(M_{,1}) = 0$,
  $\max(M_{,2}) = 0$,
  $\max(M_{,3}) = 0.07276$.  This gives
 $$N^- = \left( \begin{array}{rrr}
 -0.46575 & -0.29144 & 0       \\
 -0.17430 & -0.07276 & -0.36421\\
 -0.07276 & -0.36421 & -0.29144\\
 0        & 0        & -0.07276\\
 \end{array}\right)$$
  We can check that $N^- \leq 0$ and $LN^- = I_3$.
\item Since $B=1$, we have $\calL(B) = 0$.
  For $x = -\z^7 - \z^3 + \z^2$, we have
  $$X = \ell(x) = (1.40668, 0.65107, 1.72510, -0.94965)$$
\item
  We compute
  $$-XN^+ = (-0.42539, 0.05376, -0.36109)$$
  $$-XN^- = ( 0.89419, 1.08567,  0.67081)$$
  In this example, the only $U \in \Z^3$ within the bounds is
  $U = (0,1,0)$.
  However, for this $U$, we have
  $$UL+X = (1.07636, 2.74414, -1.17435, 0.18706)$$
  hence this is not a solution.

\end{enumerate}
  This computation shows that, in this example, Problem \ref{problem2} has no solution.\\
Analogous computations applied to all prime $p$ in the list \eqref{list} allow to give a complete answer to Question \ref{qu:continued}:
\begin{theorem}\label{teo:risposta}
Let $p$ be one of the primes for which $\QQ(\zeta_{p-1})$ has class number 1, as listed in \eqref{list}, and let $\mathfrak{P}$ be a prime ideal over $p$ in the ring of integers of $\QQ(\zeta_{p-1})$. Then $\mathfrak{P}$ is large if and only if
\begin{equation}\label{eq:listalarge} p\in \{5,7,11,13,19,31\}.\end{equation}
\end{theorem}
More precisely, for each primes in the list \eqref{eq:listalarge} the following table gives (up to Galois conjugation and multiplication by a root of unity) the elements $\pi$ of absolute norm $p$ having all the components $\geq 1$ in the canonical embedding ($\zeta=\zeta_{p-1}$ in each case): 

\begin{table}[h]
\def\arraystretch{1.5}
\begin{tabular}{l|l}    $p$ & $\pi$\\      \hline $5$ & $2\zeta+1$\\
    $7$ & $\zeta-3$\\
    $11$ & $2\zeta^3-1$, \quad $2\zeta^2-\zeta+1$\\
    $13$ & $-2\zeta^3-\zeta^2$, \quad $\zeta^3-\zeta^2+2$\\
    $19$  & $-\zeta^4-\zeta^3+\zeta^2+\zeta+1$\\
    $31$ & $-\zeta^7-\zeta^3-\zeta$,\quad $ -\zeta^6-\zeta^5+\zeta^3+\zeta^2+\zeta-1$\\
    &
    \end{tabular}
   
    \caption{Strictly large integers of absolute norm $p$ in $\QQ(\zeta_{p-1})$} \label{tavola} 
\end{table}

\section{Another application: floor functions and types}\label{sect:floorfunctions}

Let $K$ be a number field of degree $d$ over $\QQ$, and let $\mathcal{O}_K$ be its ring of integers. We fix an ideal $\frA$ of $\calo_K$.
The aim of this section is to apply largeness (when possible) in order to define  complete sets of representatives of $K/\frA$  (which will be called \emph{types}) satisfying some integrality properties and having all the Archimedean embeddings bounded in a controlled way. \\ Types associated to a prime ideal $\mathfrak{P}$ of a number field were introduced in \cite{CapuanoMurruTerracini2022} with the aim of  constructing a general notion of $\mathfrak{P}$-adic continued fractions and studying their finiteness and periodicity properties.

Let $\calM_K^0$ be a set of representatives for the non-Archimedean places of $K$. For every rational prime $p$ and every  $v\in\mathcal{M}_K^0$ above $p$ let $K_{v}\subseteq \overline{\mathbb{Q}}_p$ be the completion of $K$ w.r.t. the $v$-adic valuation and $\calo_v$ be its valuation ring; we put $d_v=[K_v:\QQ_p]$. Let $|\cdot |_v=|N_{K_v/\QQ_p}(\cdot)|_p^{\frac 1{d_v}}$ be the unique extension  of $|\cdot |_p$ to $K_v$.  
Let $\widetilde{K}=\prod_{v \mid \frA} K_v$ be the $\frA$-adic completion of $K$, with $K$ diagonally embedded, and $\widetilde{\calo}=\prod_{v \mid \frA}\calo_v$.\\
Let $S_0=\{v\in \calM_K^0 \mid  v \mid \frA\}$.
\begin{definition}
An \emph{$\frA$-adic floor function} for $K$ is a function $s:\widetilde K\to K$ such that
\begin{itemize} 
\item[a)] $\alpha-s(\alpha)\in \frA\widetilde\calo$ for every $\alpha\in \widetilde K$;
\item[b)] $|s(\alpha)|_{v}\leq 1$ for every  $v\in \calM_K^0\setminus S_0$;
\item[c)] $s(0)=0$;
\item[d)] $s(\alpha)=s(\beta)$ if $\alpha-\beta\in\frA\widetilde\calo$.
\end{itemize}
\end{definition}

By the Strong Approximation Theorem in number fields (see for example \cite[Theorem 4.1]{Cassels}),  $\frA$-adic floor functions always exist, and there are infinitely many. \\
We define the ring of $S_0$-integers
\[\calo_{K,S_0}=\{\alpha\in K\ |\ |\alpha|_v\leq 1\hbox{ for every } v \in \calM_K^0 \setminus S_0\}.
\]
Then, we can regard an $\frA$-adic floor function as a map $s: \widetilde K/\frA\widetilde\calo\to \calo_{K,S_0}$ such that $s(\frA\widetilde\calo)=0$ and which is a section of the projection map $\widetilde K\to \widetilde K/\frA\widetilde\calo$.
Therefore the choice of an $\frA$-adic floor function amounts to choose  a set $\calY$ of representatives of the cosets of $\frA\widetilde\calo$ in $\widetilde K$ containing $0$ and contained in
$\calo_{K,S_0}$.\\
We shall call the data $\tau=(K,\frA,s)$ (or $(K,\frA,\calY)$) a \emph{type}. 
\begin{remark} The absolute Galois group $\Gal(\Qalg/\QQ)$ acts on the set of types; indeed, if $\tau=(K,\frA, s)$ is a type, then $\sigma\in \Gal(\Qalg/\QQ)$ induces a continuous map $\widetilde K \to \widetilde{K^\sigma}$, where $\widetilde{K^\sigma}$  is the completion of $K^\sigma$ with respect to the ideal $\frA^\sigma$.  Then  $\tau^\sigma=(K^\sigma,\frA^\sigma, s^\sigma)$ is also a type, where $s^\sigma=\sigma\circ s\circ \sigma^{-1}$. In particular, if $K/\QQ$ is a Galois extension and $\sigma$ belongs to the decomposition group 
$$D_\frA=\{\sigma\in\Gal(\Qalg/\QQ)\ |\  \frA^\sigma=\frA\},$$
   then $\tau^\sigma=(K,\frA, s^\sigma)$ is again an $\frA$-adic type.
\end{remark}
\subsection{Types arising from generators of $\frA$} \label{sec:special_type}
In the case $\frA$ is principal, there is a natural way of defining an $\frA$-adic floor function.
Indeed, let $\pi\in\frA$ be  generator and let $\mathcal{R}$ be a  complete set of representatives of $\calo_K/\frA$ containing $0$. Then, every $\alpha\in \widetilde K$ can be expressed uniquely as a Laurent series  $\alpha=\sum_{j=-n}^\infty c_j\pi^j$, where $c_j\in\mathcal{R}$ for every $j$. It is possible to define an  $\frA$-adic floor function by
$$s(\alpha)=\sum_{j=-n}^0c_j\pi^j\in K. $$
We shall denote the types $\tau=(K,\frA,s)$ obtained in this way by $\tau=(K,\pi,\calR)$, and we will usually call them \emph{special types}.

\begin{example}[Browkin and Ruban types over $\mathbb{Q}$]
When $K=\QQ$ and $\pi=p$ odd prime, two main special types have been studied in the literature:
\begin{itemize}
\item the \emph{Browkin type} $\tau_B=(\QQ,p,\calR_B)$ where $\calR_B=\{- \frac {p-1} 2,\ldots, \frac {p-1} 2\}$ (see \cite{Browkin1978, Bedocchi1988, Bedocchi1989, Bedocchi1990, Browkin2000, CapuanoMurruTerracini2020});
\item the \emph{Ruban type} $\tau_R=(\QQ,p,\calR_R)$ where $\calR_R=\{0,\ldots, p-1\}$ (see \cite{Ruban1970, Laohakosol1985, Wang1985, CapuanoVenezianoZannier2019}).
\end{itemize}
\end{example}

\subsection{Bounded types} We say that a type $\tau=(K,\frA,s)$ is \emph{bounded} if there exists a real number $C>0$ such that  $|\sigma(s(\alpha))| <C$ for every $\alpha\in K$ and every Archimedean embedding $\sigma$ of $K$.
\begin{proposition}\label{prop:existboundedtype} For every number field $K$ and prime ideal $\frA$ there exist a bounded type $(K,\frA,s)$. 
\end{proposition}
\begin{proof}
Let $\iota: K\to \RR^{r_1}\times\CC^{r_2}\simeq \RR^d$ be the canonical embedding. Then $\iota(\frA)$ is a lattice in $\RR^d$. Let $\mathcal{D}_\frA$ be a bounded fundamental domain containing $0$. We define a floor function $s$ in the following way: firstly choose any $\frA$-adic floor function $s'$ for $K$. Let $\alpha\in \widetilde K$ and put $\alpha'=s'(\alpha)$; then $\beta= \alpha'+\gamma\in D_\frA$ for a suitable $\gamma\in \frA$; define $s(\alpha)=\beta$. Since $D_\frA$ is a bounded subset of $\RR^d$, the claim is proven.
\end{proof}
\begin{remark}  Let $\rho$ be the covering radius of the lattice $\iota(\frA)$ with respect to the sup norm. Then the closed ball centered in $0$ of radius $\rho$ with respect to this norm contains a fundamental domain $D_\frA$ as in the proof of Proposition \ref{prop:existboundedtype}. In particular we see that there exists a type $(K,\frA,s)$ such that $||\iota(s(x))||_\infty\leq \rho$, for every $x\in\widetilde K$.\end{remark} 

\begin{proposition}\label{prop:specialtypesbounded} Assume that $\frA$ is a non-zero principal ideal of $\calo_K$ having a strictly large generator $\pi$. Let $\mathcal{R}$ be any complete set of representatives of $\calo_K/\frA$ containing $0$. Then the special type $(K,\pi,\mathcal{R})$ is bounded.
\end{proposition}
 
\begin{proof}
For every Archimedean embedding $\sigma:K\to\CC$ let $\lambda_\sigma=|\sigma(\pi)|$ and $L_\sigma=\max\{|\sigma(c)|\ |\ c\in \mathcal{R}\}$. 
Then for every $\sigma$
$$\left|\sum_{j=-n}^0 \sigma(c_j)\sigma(\pi^j)\right | \leq  \frac {L_\sigma\lambda_\sigma}{\lambda_\sigma-1}.$$
\end{proof}

\begin{remark}
For each $p$ in the list \eqref{list}, and every prime ideal $\mathfrak{P}$ over $p$ in $\ZZ(\zeta_{p-1})$, the set $\mathcal{R}_p=\{\zeta^i\ |\ i=0,\ldots, p-2\}\cup\{0\}$ is a complete set of representatives of $\ZZ[\zeta]/\mathfrak{P}$. Then by Proposition \ref{prop:specialtypesbounded} the special types $(\QQ(\zeta_{p-1}),\pi,\mathcal{R}_p)$ are bounded for every $p$  and $\pi$ as in Table \ref{tavola}.
\end{remark}
\bibliographystyle{abbrv}
\bibliography{MCF}

\begin{thebibliography}{10}

\bibitem{AmorosoDavid2021}
F.~Amoroso and S.~David.
\newblock Covolumes, unit\'{e}s, r\'{e}gulateur: conjectures de {D}. {B}ertrand
  et {F}. {R}odriguez-{V}illegas.
\newblock {\em Ann. Math. Qu\'{e}.}, 45(1):1--18, 2021.

\bibitem{AmorosoDavidZannier2014}
F.~Amoroso, S.~David, and U.~Zannier.
\newblock On fields with {P}roperty ({B}).
\newblock {\em Proc. Amer. Math. Soc.}, 142(6):1893--1910, 2014.

\bibitem{AmorosoDvornicich2000}
F.~Amoroso and R.~Dvornicich.
\newblock A lower bound for the height in abelian extensions.
\newblock {\em J. Number Theory}, 80(2):260--272, 2000.

\bibitem{Bedocchi1988}
E.~Bedocchi.
\newblock A note on p-adic continued fractions.
\newblock {\em Annali di Matematica Pura ed Applicata}, pages 197--207, 1988.

\bibitem{Bedocchi1989}
E.~Bedocchi.
\newblock Remarks on periods of $p$-adic continued fractions.
\newblock {\em Bollettino dell'Unione Matematica Italiana}, 7(3-A):209--214,
  1989.

\bibitem{Bedocchi1990}
E.~Bedocchi.
\newblock Sur le developpement de $\sqrt{m}$ en fraction continue p-adique.
\newblock {\em Manuscripta Mathematica}, 67:187--195, 1990.

\bibitem{BergeMartinet1989}
A.-M. Berg\'{e} and J.~Martinet.
\newblock Notions relatives de r\'{e}gulateurs et de hauteurs.
\newblock {\em Acta Arith.}, 54(2):155--170, 1989.

\bibitem{BombieriGubler2006}
E.~Bombieri and W.~Gubler.
\newblock {\em Heights in {D}iophantine geometry}, volume~4 of {\em New
  Mathematical Monographs}.
\newblock Cambridge University Press, Cambridge, 2006.

\bibitem{Browkin1978}
J.~Browkin.
\newblock Continued fractions in local fields i.
\newblock {\em Demonstratio Mathematica}, 11:67--82., 1978.

\bibitem{Browkin2000}
J.~Browkin.
\newblock Continued fractions in local fields ii.
\newblock {\em Mathematics of Computation}, 70:1281--1292, 2000.

\bibitem{CapuanoMurruTerracini2020}
L.~Capuano, N.~Murru, and L.~Terracini.
\newblock On periodicity of $p$-adic {B}rowkin continued fractions, 2020.
\newblock https://arxiv.org/abs/2010.07364.

\bibitem{CapuanoMurruTerracini2022}
L.~Capuano, N.~Murru, and L.~Terracini.
\newblock On the finiteness of $\mathfrak{P}$-adic continued fractions for
  number fields, 2022.
\newblock to appear in the \emph{Bulletin de la Société Mathématique de
  France}.

\bibitem{CapuanoVenezianoZannier2019}
L.~Capuano, Veneziano, and U.~Zannier.
\newblock An effective criterion for periodicity of $ \ell$-adic continued
  fractions.
\newblock {\em Mathematics of Computation}, 88:1851--1882, 2019.

\bibitem{Cassels}
J.~Cassels.
\newblock {\em Local Fields (London Mathematical Society Student Texts)}.
\newblock Cambridge University Press, Cambridge, 1986.

\bibitem{FriedmanSkoruppa1999}
E.~Friedman and N.-P. Skoruppa.
\newblock Relative regulators of number fields.
\newblock {\em Invent. Math.}, 135(1):115--144, 1999.

\bibitem{Kronecker1857}
L.~Kronecker.
\newblock Zwei {S}\"{a}tze \"{u}ber {G}leichungen mit ganzzahligen
  {C}oefficienten.
\newblock {\em J. Reine Angew. Math.}, 53:173--175, 1857.

\bibitem{Laohakosol1985}
V.~Laohakosol.
\newblock A characterization of rational numbers by $p$-adic {R}uban continued
  fractions.
\newblock {\em J. Austral. Math. Soc. Ser. A}, 39:300--305, 1985.

\bibitem{MasleyMontgomery1976}
J.~Masley and H.~Montgomery.
\newblock Cyclotomic fields with unique factorization.
\newblock {\em J. Reine Angew. Math.}, 286/287:248--256, 1976.

\bibitem{MicciancioGoldwasser2002}
D.~Micciancio and S.~Goldwasser.
\newblock {\em Complexity of lattice problems}, volume 671 of {\em The Kluwer
  International Series in Engineering and Computer Science}.
\newblock Kluwer Academic Publishers, Boston, MA, 2002.
\newblock A cryptographic perspective.

\bibitem{Narkiewicz2004}
W.~a.~a. Narkiewicz.
\newblock {\em Elementary and analytic theory of algebraic numbers}.
\newblock Springer Monographs in Mathematics. Springer-Verlag, Berlin, third
  edition, 2004.

\bibitem{Ruban1970}
A.~A. Ruban.
\newblock Certain metric properties of the {$p$}-adic numbers.
\newblock {\em Sibirsk. Mat. \v{Z}.}, 11:222--227, 1970.

\bibitem{Schinzel1973}
A.~Schinzel.
\newblock On the product of the conjugates outside the unit circle of an
  algebraic number.
\newblock {\em Acta Arith.}, 24:385--399, 1973.
\newblock Addendum; ibid., {\bf 26} (1973), 329--361.

\bibitem{Silverman1984}
J.~H. Silverman.
\newblock An inequality relating the regulator and the discriminant of a number
  field.
\newblock {\em JNT}, 19:437--442, 1984.

\bibitem{PARI2}
{The PARI~Group}, Univ. Bordeaux.
\newblock {\em {PARI/GP version \texttt{2.13.4}}}, 2022.
\newblock available from \url{http://pari.math.u-bordeaux.fr/}.

\bibitem{Wang1985}
L.~X. Wang.
\newblock {$p$}-adic continued fractions. {I}, {II}.
\newblock {\em Sci. Sinica Ser. A}, 28(10):1009--1017, 1018--1023, 1985.

\bibitem{Zimmert1981}
R.~Zimmert.
\newblock Ideale kleiner {N}orm in {I}dealklassen und eine
  {R}egulatorabsch\"{a}tzung.
\newblock {\em Invent. Math.}, 62(3):367--380, 1981.

\end{thebibliography}

\end{document}